\documentclass[11pt]{article}
\usepackage{amsmath,amsthm}
\usepackage{amssymb,mathrsfs}
\usepackage{bm,mathtools}

\usepackage{xcolor}
\usepackage{geometry}
\usepackage[colorlinks=true,
linkcolor=blue,citecolor=blue,
urlcolor=blue]{hyperref}

\makeatletter
\def\@seccntDot{.}
\def\@seccntformat#1{\csname the#1\endcsname\@seccntDot\hskip 0.5em}
\renewcommand\section{\@startsection{section}{1}{\z@}%
{18\p@ \@plus 6\p@ \@minus 3\p@}%
{9\p@ \@plus 6\p@ \@minus 3\p@}%
{\large\bfseries\boldmath}}
\renewcommand\subsection{\@startsection{subsection}{2}{\z@}%
{12\p@ \@plus 6\p@ \@minus 3\p@}%
{3\p@ \@plus 6\p@ \@minus 3\p@}%
{\bfseries\boldmath}}
\renewcommand\subsubsection{\@startsection{subsubsection}{3}{\z@}%
{12\p@ \@plus 6\p@ \@minus 3\p@}%
{\p@}%
{\bfseries\boldmath}}
\makeatother

\usepackage{microtype}

\theoremstyle{plain}
\newtheorem{thm}{Theorem}

\newtheorem{lem}{Lemma}[section]

\newtheorem{proposition}{Proposition}[section]

\theoremstyle{definition}
\newtheorem{definition}{Definition}[section]
\newtheorem{remark}{Remark}[section]

\newtheorem{claim}{Claim}

\numberwithin{equation}{section}
\allowdisplaybreaks
\title{Localized and weighted versions of extremal problems}

\author{Binlong Li\footnote{School of Mathematics and Statistics, Northwestern Polytechnical University, Xi’an, Shaanxi 710072,
 P. R. China. Supported by NSFC (No. 12071370) and Shaanxi Fundamental Science Research Project
 for Mathematics and Physics (No. 22JSZ009). Email: \texttt{binlongli@nwpu.edu.cn.}}~and~Bo Ning\footnote{Corresponding author. College of Computer Science, Nankai University, Tianjin 300350, P.R. China.
Partially supported by the National Nature Science
Foundation of China (Nos. 12371350 and 12426675) and Fundamental Research Funds for the Central Universities, Nankai University (No. 63243151). Email: \texttt{bo.ning@nankai.edu.cn} (B. Ning).}}
\date{}

\begin{document}
\maketitle

\begin{abstract}
\par\vspace{2mm} Malec and Tompkins (EUJC, 2023)
considered the localized versions of Tur\'an-type problems, and proved a localized
theorem on Erd\H{o}s-Gallai Theorem on paths. Zhao and Zhang (JGT, 2025) gave a long
proof of a localized version of Erd\H{o}s-Gallai Theorem on cycles.

In this paper, we consider several types of generalization of
Tur\'an-type problems, that is, localized versions, weighted versions, and generalized Tur\'an-type
problems, and their connectedness.
We first present very short proofs for recent results of Malec-Tompkins and Zhao-Zhang, respectively.
We use Small Path Double Cover Conjecture,
which was proposed by Bondy (JGT, 1990) and confirmed by Hao Li (JGT, 1990), to prove a weighted localized Tur\'an-type
theorem on paths.
We prove localized versions of Balister-Bollob\'as-Riordan-Schelp
Theorem (JCTB, 2003) on paths and Erd\H{o}s-Gallai Theorem on matchings, respectively. We show that our first localized result implies Balister-Bollob\'as-
Riordan-Schelp Theorem, Erd\H{o}s-Gallai Theorem, and Malec-Tompkins Theorem on paths.
Finally, we present generalized Tur\'an-style generalizations of the Malec-Tompkin's Theorem, and discuss the relationship between some previous theorems in
different motivations.

\noindent{\bfseries Keywords:~}paths and cycles; matchings; weighted theorem; Tur\'an-type problem; small path double cover conjecture
\par\vspace{2mm}

\noindent{\bfseries AMS Classification: 05C38; 05C35}
\end{abstract}
\section{Introduction}

In 1959, Erd\H{o}s and Gallai \cite{EG1959}  proved three theorems
on paths, cycles, and matchings, which are milestone results in extremal graph theory.

\begin{thm}(Erd\H{o}s-Gallai \cite{EG1959})\label{Thm:EG}
Let $G$ be an $n$-vertex graph with $m$ edges.\\
(A) Then $G$ contains a path of length at least
$\frac{2m}{n}$.\\
(B) If $G$ is 2-edge-connected, then $G$ contains a cycle of length at least
$\frac{2m}{n-1}$.\\
(C) Suppose that $G$ has matching number $\mu=k$. If $n\geq 2k+1$ then
$$
e(G)\leq \max\left\{\binom{2k+1}{2},\binom{k}{2}+(n-k)k\right\}.
$$
\end{thm}

Balister, Bollob\'as, Riordan, and Schelp
\cite{BBRS2003} generalized Theorem \ref{Thm:EG}(A) as follows.
\begin{thm}[Balister-Bollob\'as-Riordan-Schelp \cite{BBRS2003}]\label{Thm:BBRS2003}
For every graph $G$ we have
$$
e(G)\leq \frac{1}{2}\sum_{v\in V(G)}\ell_G(x),
$$
where $\ell_{G}(x)$ is the length of a longest path starting
at $x\in V(G)$, with equality if and only if every component of G is a complete graph.
\end{thm}

Our point of departure is the following theorems recently proved by Malec and Tompkins \cite{MT2023} and
Zhao and Zhang \cite{ZZ25}, respectively.

\begin{thm}
(A) [Malec-Tompkins \cite{MT2023}]\label{Theorem:MT}
Let $G$ be an $n$-vertex graph. Then
$$\sum_{e\in E(G)}\frac{1}{p(e)}\leq \frac{n}{2},$$
where $p(e)$ denotes the length of a longest path containing $e$.\\
(B) [Zhao-Zhang \cite{ZZ25}]\label{Theorem:Zhang}
Let $G$ be an $n$-vertex graph. Then
$$\sum_{e\in E(G)}\frac{1}{c(e)}\leq \frac{n-1}{2},$$
where $c(e)$ denotes the length of a longest cycle
containing $e$ if $e$ belongs to some cycle in $G$;
and 2 if $e$ is not contained in some cycle.
\end{thm}

Rather than forbidding certain substructures in a graph, they defined
a weighted function that can be used to arbitrary structures. For a family of
graphs $\mathcal{H}:=\{H_k:1\leq n\}$ (usually they are of the same kind of graphs, for example,
they are all paths, cycles, or other graphs), define $f_{\mathcal{H}}(e)$ to be the largest size of $H_i\in \mathcal{H}$
such that $e$ is an edge of some subgraph of $G$ isomorphic to $H_i$. Then they aim to control the upper bound of some function, such as $\sum_{e\in E(G)}g(f_{\mathcal{H}}(e))$, where $g(\cdot)$ is a function. This kind of extremal problem is called ``the localized version of extremal problems" in \cite{MT2023}. Next, we give an explanation why the Malec-Tompkins Theorem implies Erd\H{o}s-Gallai Theorem on paths.
Theorem \ref{Thm:EG}(A) is easily re-obtained by taking a $P_k$-free graph $G$
and using $p(e)\leq k-1$. Similar ideas can help us to recover Theorem \ref{Thm:EG}(B) by using the
Zhao-Zhang Theorem.
For more results in this
direction, we refer the reader to \cite{KN24,ZZ25-2,LN-arxiv23} and references therein.

In this paper, we first present localized versions of Balister-Bollob\'as-Riordan-Schelp Theorem
on paths and Erd\H{o}s-Gallai Theorem on matchings. We shall show that Theorem \ref{Th:Local-BBRS} implies both Theorem \ref{Thm:BBRS2003}
and Theorem \ref{Theorem:MT}(A).

\begin{thm}\label{Th:Local-BBRS} 
Let $G$ be a connected $n$-vertex graph and $v\in V(G)$. For any edge $e\in E(G)$,
denote by $p_v(e)$ the length of a longest path starting at $v$ and going through $e$
and by $E_v(G)=\{e: e\in E(G), e\mbox{~is~incident~to~v}\}$.
Then, we have
$$\frac{1}{2}\sum_{e\in E_v(G)}\frac{1}{p_v(e)}+\sum_{e\notin E_v(G)}\frac{1}{p_v(e)}\leq \frac{n-1}{2}.$$
Equality holds if and only if $G$ consists of (at least two) cliques sharing $v$ as a common vertex when $v$ is a cut-vertex; and $G$ is a clique when $v$ is not a cut-vertex.
\end{thm}

\begin{thm}\label{Th:Local-Matching}
  Let $G$ be an $n$-vertex graph with matching number $\mu$. For each $e\in E(G)$, denote by $\mu(e)$ the size of a maximum matching containing $e$.
 
  (1) If $n=2\mu$ (i.e., $G$ has a perfect matching), then 
  $$\sum_{e\in E(G)}\frac{1}{\mu(e)}\leq n-1.$$ Equality holds if and only if $G$ is a complete graph.
  
  (2) If $n\geq 2\mu+1$, then $$\sum_{e\in E(G)}\frac{1}{\mu(e)}\leq\left\{\begin{array}{ll}
    \max\{3,n-1\},  & \mu=1;\\
    \max\{2\mu+1,n-\mu/2\}, & \mu\geq 2.

\end{array}\right.$$
For $\mu=1$, equality holds if and only if $n=3$ and $G\cong K_3$, or $n=4$ and $G\cong K_3\cup K_1$, or $n\geq 4$ and $G\cong K_{1,n-1}$; for $\mu\geq 2$, equality holds if and only if $n\leq 5\mu/2+1$ and $G\cong K_{2\mu+1}\cup\overline{K_{n-2\mu-1}}$, or $n\geq 5\mu/2+1$ and $G\cong K_{\mu}\vee\overline{K_{n-\mu}}$.
\end{thm}

In this paper, we prove a weighted localized Tur\'an-type
theorem on paths, which generalizes  Theorem \ref{Thm:EG}(A) and Theorem \ref{Theorem:MT}(A) on paths. We shall present very
short proofs of Malec-Tompkins Theorem and Zhao-Zhang Theorem, respectively.
Note that Zhao-Zhang's proof of Theorem \ref{Theorem:Zhang}(B) is very
complicated and has more than 18 pages.
(The interested reader can directly go to Section \ref{Sec:2} to read our proofs.)

\begin{thm}\label{Thm:Weight-MT}
Let $G$ be a weighted $n$-vertex graph. Then
$$
\sum_{e\in E(G)}\frac{w(e)}{w(p(e))}\leq \frac{n}{2},
$$
where $w(p(e))$ denotes the maximum weight of a path
containing the edge $e$.
\end{thm}

An example explaining our strategy is presented below.
\begin{itemize}
\item Bondy's Small Path Double Cover Conjecture $\Longrightarrow$ Theorem \ref{Thm:Weight-MT} (our new one) $\Longrightarrow$ Theorem \ref{Thm:FMR}(A) (weighted version) $\Longrightarrow$ Theorem \ref{Theorem:MT}(A) (localized version).
\end{itemize} 

Motivated by Luo's work \cite{L18}, Ning and Peng \cite{NP2020} proved the following extension of Erd\H{o}s-Gallai Theorem on paths:
Let $G$ be a graph. For each integer $1\leq s\leq \omega(G)$, $G$ contains a path of
length at least $\frac{(s+1)n_{s+1}(G)}{n_s(G)}+s-1$, where $\omega(G)$ is clique number
of $G$, and $n_s(G)$ is the number of $s$-cliques in $G$. Motivated by \cite{L18,NP2020} , in this paper
we finally present a generalized Tur\'an-style generalization of the Malec-Tompkins Theorem,
which is a mixture of ideas from generalized Tur\'an-type theorem and local Tur\'an-type
theorem. The cases $s=2$ of the following theorems reduce to Theorem \ref{Theorem:MT}(A)
and Proposition 1 in \cite{MT2023}, respectively.

\begin{thm}\label{Thm:Local-generalized-Turan}
Let $G$ be a graph. Then\\
(i)
$$\sum_{S\in N_s(G)}\frac{1}{p(S)-s+2}\leq \frac{n_{s-1}}{s},$$
where $p(S)$ denotes the maximum length of a path $P$ such that $S\subset V(P)$
and the vertices of $S$ appear on $P$ in consecutive order.\\
(ii) We have
$$\sum_{K\in N_s(G)}\frac{1}{s(K)-s+2}\leq \frac{n_{s-1}}{s},$$
where $s(K)$ denotes the maximum size of a star $S$ (in $G$) such that a clique $K\subset V(S)$.
\end{thm}

We introduce some notation. Let $G=(V(G),E(G))$ be a simple graph (without loops or multiple edges), where $V(G)$ and $E(G)$ denote the vertex set and edge set of $G$. We say that $G$ is
a \emph{weighted graph} if each edge $e$ is assigned a non-negative number $w(e)$, which is called the
\emph{weight} of $e$; if we want to distinguish the underlying graph, we write $w_G(e)$
instead of $w(e)$. For any subgraph $H$ of $G$, the \emph{weight} of $H$ is defined as
$w(H)=\sum_{e\in E(H)}w(e)$. Setting $w(e)=1$ for any edge $e\in E(G)$ in the weighted graph $G$,
one can see that a weighted graph is a generalization of the ordinary graph. Let $S\subseteq V(G)$. Denote by $G[S]$ the
subgraph of $G$ induced by $S$. We use $G-S$ to denote the induced subgraph $G[V(G)\backslash S]$.
When $S=\{v\}$, we use $G-v$ instead of $G-\{v\}$. Let $H_1$ and $H_2$ be two graphs. We use $H_1\vee H_2$ to denote the join of $H_1$ and $H_2$. For a subgraph $H$ of $G$, the size of $H$ is defined
as the number of edges in $H$. When we say that $H$ is a subgraph of $G$, we use the notation $H\subseteq G$. Let $s\geq 1$
be an integer. We use $N_s(G)$ to denote the set of all $s$-cliques in $G$, and $n_s(G)=|N_s(G)|$.

The rest of the paper is organized as follows. In Section \ref{Sec:2},
we give very short proofs of Theorem \ref{Theorem:MT}(A) and (B), and Theorem \ref{Thm:Weight-MT}, respectively.
In Section \ref{Sec:3}, 
we present proofs of Theorems
\ref{Th:Local-BBRS}, \ref{Th:Local-Matching}, and \ref{Thm:Local-generalized-Turan},
respectively.

\section{Short proofs of Theorem
\ref{Theorem:MT}(A) and Theorem \ref{Theorem:MT}(B), and proof of Theorem \ref{Thm:Weight-MT}}\label{Sec:2}
Our main tool for short proofs of Theorems
\ref{Theorem:MT}(A) and (B) is the following theorem.

\begin{thm}\label{Thm:weighted graphs}
(A) (Frieze-McDiarmid-Reed \cite{FMR1992})\label{Thm:FMR}
Let $G$ be a weighted graph. Then $G$ has a path with weight at least $$\frac{2w(G)}{n}.$$\\
(B) (Bondy-Fan \cite{BF1991})\label{Thm:Bondy-Fan1991}
Let $G$ be a 2-edge-connected weighted graph. Then $G$ has a cycle with weight at least $$\frac{2w(G)}{n-1}.$$
\end{thm}

\noindent
{\bf A short proof of Theorem \ref{Theorem:MT}(A).}
Set $w(e)=\frac{1}{p(e)}$. Then by Theorem \ref{Thm:FMR}(A), there is a path $P'$ with weight $w(P')\geq \sum_{e\in E(G)}\frac{2}{p(e)\cdot n}.$
For any path $P$ in $G$, we have
$w(P)=\sum_{e\in E(P)}\frac{1}{p(e)}\leq \sum_{e\in E(P)}\frac{1}{e(P)}=1,$
as $p(e)\geq e(P)$ for any $e\in E(P)$.
Thus, $w(P)\leq 1.$ 
The proof is complete. 
\qed

\vspace{2mm}

\noindent
{\bf A short proof of Theorem \ref{Theorem:Zhang}(B).}

Set
$c(e)=\max\{k: e\in E(H), H\subseteq G, H\mbox{ is isomorphic to~} C_k\}.$
First assume that $G$ is 2 edge-connected.
We construct a weighted 2 edge-connected $n$-vertex graph $G'$, which is obtained from $G$
by weighting each $e$ with the weight $\frac{1}{c(e)}$. Then $w(G')=\sum_{e\in E(G)}\frac{1}{c(e)}$.
By Theorem \ref{Thm:Bondy-Fan1991}(B), $G'$ has a cycle $C$ of weight at least $\frac{2w(G')}{n-1}=\sum_{e\in E(G)}\frac{2}{c(e)(n-1)}$.
Denote by $C'$ a cycle of $G$ with the maximum weight. Then, $w(C')=\sum_{e\in E(C')}w(e)=\sum_{e\in E(C')}\frac{1}{e(C)}=1$.
Thus,
$1\geq w(C)\geq \sum_{e\in E(G)}\frac{2}{c(e)(n-1)}.$

Next, consider the case $G$ is not 2 edge-connected. Suppose
$G$ contains $t$ cut-edges, in which we delete all these edges
resulting in $t+1$ components and each is 2 edge-connected.
Let $G_1,G_2,\ldots,G_{t+1}$ be all these components,
and $n_i=|G_i|$, where $i=1,2,\ldots,t$. Thus,
$\sum_{e\in E(G_i)}\frac{1}{c(e)}\leq \frac{n_i-1}{2}$ for any $i$.
Taking into account all these inequalities, we have
$$\sum_{e\in E(G)}\frac{1}{c(e)}=\sum_{i=1}^{t+1}\sum_{e\in E(G_i)}\frac{1}{c(e)}+\frac{t}{2}\leq \frac{n-(t+1)}{2}+\frac{t}{2}=\frac{n-1}{2}.$$
\qed
\begin{remark}
One can determine extremal graphs in Theorem \ref{Theorem:MT}(A) using Theorem 6.2 in \cite{BF1991}. For the extremal graphs in Theorem \ref{Theorem:Zhang}(B), it is more difficult. We refer the reader to Section 3 in \cite{BF1991}.
\end{remark}

Before the proof of Theorem \ref{Thm:Weight-MT}, we present two remarks.

\begin{remark}
Set $w(e)=1$ in $G$. Then $G$ is a simple $n$-vertex graph, and $w(p(e))=p(e)$. So, Theorem \ref{Theorem:MT} is a
corollary of Theorem \ref{Thm:Weight-MT}.
\end{remark}

\begin{remark}
We show that Theorem \ref{Thm:Weight-MT} implies Theorem \ref{Thm:FMR}(A). Choose a path $P^{*}\subseteq G$ so that $w(P^{*})$ is the maximum. For any $e\in E(G)$,
we have $w(p(e))\leq w(P^{*})$. Then
$$
\sum_{e\in E(G)}\frac{w(e)}{w(p(e))}\geq \sum_{e\in E(G)}\frac{w(e)}{w(P^{*})}.
$$
By Theorem \ref{Thm:Weight-MT}, we have $\sum_{e\in E(G)}\frac{w(e)}{w(p(e))}\leq \frac{n}{2}$.
Together with the above two inequalities, we have
$$
w(P^{*})\geq \frac{2w(G)}{n}.
$$
\end{remark}

Our novel idea to prove Theorem \ref{Thm:Weight-MT} is to use
the famous Small Path Double Cover Conjecture. This conjecture
was proposed by Bondy \cite{B1990} and confirmed by H. Li \cite{L1990} in 1990.
\begin{definition}[\cite{L1990}]
(i) A path double cover (PDC) of a graph $G$ is a collection $\mathcal{P}$
of paths of $G$ such that each edge of $G$ belongs to exactly two
paths of $\mathcal{P}$.\\
(ii) A small path double cover (SPDC) of a graph $G$ on $n$ vertices is a
PDC $\mathcal{P}$ of $G$ such that $|\mathcal{P}|\leq n$.
\end{definition}

\begin{thm}[H. Li \cite{L1990}]\label{Thm:LiHao}
Every simple graph $G$ admits a small path double cover.
\end{thm}

\noindent
{\bf Proof of Theorem \ref{Thm:Weight-MT}}.
By Theorem \ref{Thm:LiHao}, there is a set of paths $\mathcal{P}=\{P_i:1\leq i\leq t\}$, where $t\leq n$,
such that each edge is contained in exactly two paths of $P$. Thus, we have
$$
\bigcup_{i=1}^n E(P_i)=\{2e_i: e_i\in E(G)\}.
$$
Observe that
$\sum_{e\in E(P_i)}w(e)=w(P_i)$
and $w(P(e))\geq w(P_i)$ when $e\in E(P_i)$
for any $i\in [t]$.
So
$$
\sum_{e\in E(G)}\frac{2w(e)}{w(p(e))}=\sum_{i=1}^t\sum_{e\in E(P_i)}\frac{w(e)}{w(p(e))}\leq \sum_{i=1}^t\sum_{e\in E(P_i)}\frac{w(e)}{w(P_i)}=t\leq n,
$$
and it follows $\sum_{e\in E(G)}\frac{w(e)}{w(p(e))}\leq \frac{n}{2}$. \qed

\section{Proofs of Theorems
\ref{Th:Local-BBRS}, \ref{Th:Local-Matching}, and \ref{Thm:Local-generalized-Turan}}\label{Sec:3}

\vspace{2mm}
\noindent
{\bf Proof of Theorem \ref{Th:Local-BBRS}.}
We prove the theorem by induction on $n$. The assertion is trivial if $n=2$. So we assume $n\geq 3$.

\setcounter{claim}{0}
\begin{claim}\label{Clnotcut}
If $v$ is a cut-vertex of $G$, then
$$\frac{1}{2}\sum_{e\in E_v(G)}\frac{1}{p_v(e)}+\sum_{e\notin E_v(G)}\frac{1}{p_v(e)}\leq \frac{n-1}{2},$$ 
where equality holds if and only if $G$ consists of (some) cliques sharing $v$ as a common vertex.
\end{claim}

\begin{proof}
Suppose that $v$ is a cut-vertex of $G$. Let $H_1,\ldots,H_k$ be all components of $G-v$,
$G_i=G[V(H_i)\cup\{v\}]$ and $n_i=|V(G_i)|$, where $i\in [k]$.
Observe that for every edge $e\in E(G_i)$, a longest $v$-path through $e$ in $G$ is also such
a path in $G_i$. This means that the length of a longest $v$-path through any edge $e\in G_i$ in $G_i$ is the same as the length of a longest $v$-path through any edge $e\in G_i$ in $G$.
Denote by $E_v(G_i)=\{e:e\in E(G_i),\mbox{e is incident to v}\}$.
By the induction hypothesis,
$$\frac{1}{2}\sum_{e\in E_v(G_i)}\frac{1}{p_v(e)}+\sum_{e\in E(G_i)\backslash E_v(G_i)}\frac{1}{p_v(e)}\leq \frac{n_i-1}{2},$$
where equality holds if and only if $G_i$ is a clique when $v$ is not a cut-vertex of $G_i$; and $G_i$ consists of cliques sharing the common vertex $v$, when $v$ is the cut-vertex. As $H_i:=G_i-v$ is a component, equality of the inequality on $G_i$ holds if and only if $G_i$ is a clique.
Taking the sum of all $i=1,\ldots,k$, we can get the assertion.
\end{proof}
We assume that $v$ is not a cut-vertex of $G$.
Denote by $p$ the length of a longest $v$-path of $G$. We have the following.

\begin{claim}\label{ClLogestvPath}
If $u$ is the terminus of a longest $v$-path of $G$, then $p_v(e)=p$ for every edge $e\in E_u(G)=\{uw:w\in N_G(u)\}$ and $d(u)>(p+\varepsilon)/2$, where $\varepsilon=1$ if $vu\in E(G)$, and $\varepsilon=0$ otherwise.
\end{claim}

\begin{proof}
Let $P$ be a longest $v$-path of $G$ with the terminus $u$, and let $G'=G-u$. We first show that $G'$ is connected. We have $N(u)\subseteq V(P)$ as $P$ is a longest $v$-path. If $u$ is a cut-vertex of $G$, then $G'$ has a component not containing $v$. Any neighbor of $u$ in such component is not contained in $P$, a contradiction. Thus, $G'$ is connected.

For every edge $e\in E_u(G)$, say $e=uw$ with $w\in V(P)$, the path $P'=P[v,w]wuP[u,w^+]$ (where $w^+$ is the successor of $w$ along $P$ in the direction from $v$ to $u$) is a longest $v$-path through $e$. That is, $p_v(e)=|E(P')|=p$.

Denote by $p'_v(e)$ the length of a longest $v$-path of $G'$ through $e$ for $e\in E(G')=E(G)\backslash E_u(G)$. Clearly $p_v(e)\geq p'_v(e)$ for every edge $e\in E(G')$. By the induction hypothesis,
$$\frac{1}{2}\sum_{e\in E_v(G')}\frac{1}{p'_v(e)}+\sum_{e\in E(G')\backslash E_v(G')}\frac{1}{p'_v(e)}\leq \frac{n-2}{2}.$$
If $d(u)\leq(p+\varepsilon)/2$, then
\begin{align*}
  &\frac{1}{2}\sum_{e\in E_v(G)}\frac{1}{p_v(e)}+\sum_{e\notin E_v(G)}\frac{1}{p_v(e)}\\
  =    &\frac{1}{2}\sum_{e\in E_v(G')}\frac{1}{p_v(e)}+\frac{\varepsilon}{2p_v(vu)}
        +\sum_{e\in E(G')\backslash E_v(G')}\frac{1}{p_v(e)}+\sum_{e\in E_u(G)\backslash E_v(G)}\frac{1}{p_v(e)}\\
  \leq &\frac{1}{2}\sum_{e\in E_v(G')}\frac{1}{p'_v(e)}+\sum_{e\in E(G')\backslash E_v(G')}\frac{1}{p'_v(e)}
        +\frac{\varepsilon}{2p}+\sum_{e\in E_u(G)\backslash E_v(G)}\frac{1}{p}\\
  \leq &\frac{n-2}{2}+\frac{1}{p}\left(d(u)-\frac{\varepsilon}{2}\right)\\
  \leq &\frac{n-2}{2}+\frac{1}{p}\left(\frac{p+\varepsilon}{2}-\frac{\varepsilon}{2}\right)\\
  \leq&\frac{n-1}{2},
\end{align*}
as desired. Thus, we have $d(u)>(p+\varepsilon)/2$.
\end{proof}

Now, let $P$ be a longest $v$-path of $G$, and let $u$ be the terminus of $P$. Set $P=v_0v_1v_2\ldots v_p$, where $v_0=v$ and $v_p=u$. Let $v_i\in N(u)$
such that $i$ is the minimum with respect to $P$. Subject to this, we choose $P$ so that $i$ is as small as possible among all longest $v$-paths. Let $H=G[\{v_i,v_{i+1},\ldots,v_p\}]$, and $h=|V(H)|$.

\begin{claim}\label{ClHamiltonH}
For any vertex $v_j$ with $p\geq j\geq i+1$, we have (1) $H$ has a Hamilton path from $v_i$ to $v_j$; (2) $d(v_j)>h/2$; and (3) $N(v_j)\subseteq V(H)$.
\end{claim}

\begin{proof}
We first show that for every $v_j$ with $j\in[i+1,p]$, (1) implies (2) and (3). Suppose $Q$ is a Hamilton path of $H$ from $v_i$ to $v_j$. Note that $P'=P[v,v_i]Q[v_i,v_j]$ is a longest $v$-path of $G$. By Claim \ref{ClLogestvPath}, $d(v_{j})>\frac{p+\varepsilon}{2}\geq h/2$, and (2) holds. As $P'$ is a longest $v$-path, $v_j$ has no neighbors outside $P'$; and by the choice of $P$, $v_j$ has no neighbors in $P[v_0,v_{i-1}]$. That is, $N(v_j)\subseteq V(H)$ and (3) holds.

Now it suffices to prove that $H$ has a Hamilton path from $v_i$ to $v_j$ for all $j\in[i+1,p]$. This is trivially true for $v_j=v_p$. So we suppose that $p\geq i+2$ and $j\in[i+1,p-1]$.

We prove this by induction on $j$. For $j=i+1$, $Q=v_iv_pv_{p-1}\ldots v_{i+1}$ is a Hamilton path in $H$ from $v_i$ to $v_{i+1}$. Now assume that the assertion is true for $v_{j-1}$ and we consider the vertex $v_j$, $j\leq p-1$. We notice that (2) and (3) hold for $v_{j-1}$, i.e., $d(v_{j-1})\geq(p+1)/2\geq h/2$, and $N(v_{j-1})\subseteq V(H)$. If $v_0v_p\in E(G)$, then $h=p+1$ and by Claim \ref{ClLogestvPath}, $d(v_p)>(p+1)/2=h/2$; if $v_0v_p\notin E(G)$, then $h\leq p$ and $d(v_p)>p/2\geq h/2$. For each case, we have that $d(v_{j-1})+d(v_p)>h$.

If $v_{j-1}v_p\in E(G)$, then $Q=P[v_i,v_{j-1}]v_{j-1}v_pP[v_p,v_j]$ is a Hamilton path of $H$ from $v_i$ to $v_j$, as desired. So we assume that $v_{j-1}v_p\notin E(G)$. Recall that $d(v_{j-1})+d(v_p)>h$ and $N(v_{j-1})\cup N(v_p)\subseteq V(H)$. There is either a vertex $v_k$, $k\in[i,j-2]$, such that $v_kv_{j-1}, v_{k+1}v_p\in E(G)$, or a vertex $v_k$, $k\in[j,p]$, such that $v_{j-1}v_k,v_{k-1}v_p\in E(G)$. Let $Q=P[v_i,v_k]v_kv_{j-1}P[v_{j-1},v_{k+1}]v_{k+1}v_pP[v_p,v_j]$ for the former and $Q=P[v_iv_{j-1}]v_{j-1}v_kP[v_k,v_p]$ $v_pv_{k-1}P[v_{k-1},v_j]$ for the later. Then $Q$ is a Hamilton path of $H$ from $v_i$ to $v_j$.
\end{proof}

By Claim \ref{ClHamiltonH}, every vertex in $V(H)\backslash\{v_i\}$ is the terminus of a longest $v$-path of $G$; and by Claim \ref{ClLogestvPath},
every edge $e\in E(H)$ has $p_v(e)=p$. Let $G'=G-(V(H)\backslash\{v_i\})$.

If $i=0$, then $v_0v_p=uv\in E(G)$.
By Claim \ref{ClHamiltonH}(i), 
$H$ has a Hamilton path from $v_0$ to $v_j$ for any $j\in [1,p]$. Thus, $V(H)=V(G)$; since otherwise, there is a longer $v$-path, a contradiction. It follows that $G=H$ and $p=n-1$. Thus, we have 
$$\frac{1}{2}\sum_{e\in E_v(G)}\frac{1}{p_v(e)}+\sum_{e\notin E_v(G)}\frac{1}{p_v(e)}=\frac{d(v)}{2p}+\frac{e(G)-d(v)}{p}\leq\frac{p}{2p}+\frac{p(p-1)}{2p}=\frac{p}{2}=\frac{n-1}{2},$$
where the last inequality holds as $2e(G)-d(v)\leq (n-1)^2$. Moreover, equality holds if and only if the last inequality becomes equality, and so $d(v)=p$ and $2e(G)=n(n-1)$, which implies $G$ is a $K_n$.

Now assume that $i\in[1,p-1]$. This implies that $vu\notin E(G)$, $h\leq p$, and from Claim \ref{ClHamiltonH} we know that $v_i$ is a cut-vertex of $G$. Denote by $p'_v(e)$ the length of a longest $v$-path of $G'$ through $e$ for $e\in E(G')$. Notice that $p_v(e)\geq p'_v(e)$ for every edge $e\in E(G')$. By the induction hypothesis,
$$\frac{1}{2}\sum_{e\in E_v(G')}\frac{1}{p'_v(e)}+\sum_{e\in E(G')\backslash E_v(G')}\frac{1}{p'_v(e)}\leq \frac{n-h}{2}.$$
It follows that
\begin{align*}
  &\frac{1}{2}\sum_{e\in E_v(G)}\frac{1}{p_v(e)}+\sum_{e\notin E_v(G)}\frac{1}{p_v(e)}\\
  =    &\frac{1}{2}\sum_{e\in E_v(G')}\frac{1}{p_v(e)}+\sum_{e\in E(G')\backslash E_v(G')}\frac{1}{p_v(e)}+\sum_{e\in E(H)}\frac{1}{p_v(e)}\\
  \leq &\frac{1}{2}\sum_{e\in E_v(G')}\frac{1}{p'_v(e)}+\sum_{e\in E(G')\backslash E_v(G')}\frac{1}{p'_v(e)}+\sum_{e\in E(H)}\frac{1}{p}\\
  \leq &\frac{n-h}{2}+\frac{1}{p}\cdot\frac{h(h-1)}{2}\leq\frac{n-1}{2},
\end{align*}
as desired. Notice that if all above inequalities become equalities, then $G'$ and $H$ are cliques, $p=h$, and $p_v(e)=p'_v(e)$ for all $e\in E(G')$. However, 
for an edge $e_i$ in $G'$
incident to $v_i$,
$p'_v(e_i)=i$ and $p_v(e_i)=n-1$. Thus, $p_v(e_i)\neq p'_v(e_i)$.
So, for this case, the last inequality is a strict inequality.
The proof is complete. \qed

We mention that the following result is stronger than Balister et al.'s result.

\begin{thm}[\cite{N2011}]\label{Thm:Ning-MSC}
Let $G$ be a connected graph and $v\in V(G)$. Then $G$ contains a $v$-path of length at least 
$$
\frac{2e(G)-d(v)}{n-1}.
$$
\end{thm}

\begin{remark}\label{Rem:Theorem-NingMsc-imply-ThmBBRS}
(i) We show that Theorem \ref{Thm:Ning-MSC} can imply Theorem \ref{Thm:BBRS2003}.
Recall $\ell_G(v)$ is the length of a longest $v$-path in $G$. By Theorem \ref{Thm:Ning-MSC}, $\ell_{G}(v)\geq \frac{2e(G)-d_G(v)}{n-1}$. Summing all these inequalities over all vertices $v\in V(G)$, we obtain 
$$
\sum_{v\in V(G)}\ell_{G}(v)\geq 2e(G).
$$
If $G$ is disconnected, then we use above inequality to each component.\\
(ii)
We show that Theorem \ref{Th:Local-BBRS} can imply Theorem \ref{Thm:Ning-MSC}, and hence imply Theorem \ref{Thm:BBRS2003} by Remark \ref{Rem:Theorem-NingMsc-imply-ThmBBRS}(i).
Notice that $p_v(e)\leq \ell_G(v)$ for any $e\in E(G)$. By Theorem \ref{Th:Local-BBRS}, we have
$$
\frac{1}{2}\sum_{e\in E_v(G)}\frac{1}{p_v(e)}+\sum_{e\notin E_v(G)}\frac{1}{p_v(e)}\geq \frac{e(G)-\frac{1}{2}d(v)}{\ell_G(v)}.
$$
It follows that
$\frac{e(G)-\frac{1}{2}d(v)}{\ell_G(v)}\leq \frac{n-1}{2},$
which implies
$\ell_G(v)\geq \frac{2e(G)-d(v)}{n-1}$.\\
(iii)
We show that Theorem 
\ref{Th:Local-BBRS} 
implies Theorem \ref{Theorem:MT}(A). As $p_v(e)\leq p(e)$ for any $v\in V(G)$ and $e\in E(G)$,
we have
$$
\frac{n-1}{2}\geq \frac{1}{2}\sum_{e\in E_v(G)}\frac{1}{p_v(e)}+\sum_{e\notin E_v(G)}\frac{1}{p_v(e)}\geq \frac{e(G)-\frac{1}{2}d(v)}{p(e)}.
$$
Summing over all vertices $v\in V(G)$, we obtain
$$
\frac{e(G)}{p(e)}\leq \frac{n}{2}.
$$
As $e$ is chosen arbitrarily, we have
$$
\sum_{e\in E(G)}\frac{1}{p(e)}\leq \frac{n}{2}.
$$
\end{remark}

Next, we shall present a proof of Theorem \ref{Th:Local-Matching}. The proof needs Gallai-Edmonds Structure Theorem and Bondy-Chv\'atal closure theorem on matchings.
\begin{thm}[\cite{G63,G64,E65}]\label{ThMatchingFactorCritical}
 (Gallai–Edmonds Structure Theorem). Let $G= (V,E)$ be a graph with $GE(G) =
 (C,A,D)$. Then each component of $D$ is factor-critical. Also, any maximum matching of $G$ contains
 a perfect matching of C and a near-perfect matching of each component of $D$, and matches all
 vertices of $A$ with vertices in distinct components of $D$. 
\end{thm}

\begin{lem}[Bondy and Chv\'{a}tal~\cite{BC76}]\label{matching-closure}
Let $G$ be a graph on $n$ vertices. For any two non-adjacent
vertices $u,v\in V(G)$, if whenever matching number $\mu(G+uv)=k+1$
and $d_G(u)+d_G(v)\geq 2k+1$, then $\mu(G)=k+1$.
\end{lem}
Let $G$ be a graph. 
Fix a positive integer $k \geq 1$. Perform the following graph operation: if there are two non-adjacent
vertices $u$ and $v$ with $d_G(u)+d_G(v)\geq k$, add the new edge $uv$ to $E(G)$. A $k$-\emph{closure} of $G$ is a graph obtained
from $G$ by successively applying this operation as long as possible. It is known that the $k$-closure of $G$ is unique. 

The proof of Theorem 1.1 in \cite{NW20} in fact implies the following stronger
form of lemma.
\begin{lem}\label{Lem:NW}
Let $G$ be an $n$-vertex graph with matching number $\mu=k$. Let $\Gamma$ be
the $(2k+1)$-closure of $G$. Let $C$ be the set of all vertices in $G'$ with
degree at least $k+1$, and $S$ be the set of
vertices in a maximal clique in $G'$ that contains $C$.
Denote $s=|S|$. Set $f(s)=\binom{s}{2} +(2k-s+1)(n-s)$. If $s\leq k$ then
$e(\Gamma)\leq f(k)$. If $k+1\leq s\leq t\leq 2k+1$,
then $e(\Gamma)\leq \max\left\{f(t), f(k+1)\right\}.$
\end{lem}

\begin{lem}\label{Lem:Matching-Stab}
Let $G$ be a graph with order $n$ and matching number $\mu=k$. If $n\leq \frac{5k+1}{2}$ and $e(G)>2k^2$, then $G$ is a subgraph of $K_{2k+1}$ (possibly together with some isolated vertices).
\end{lem}
\begin{proof}
We use the argument from \cite{NW20} (see the proof of Theorem 1.1 in \cite{NW20}).
Let $G$ be an $n$-vertex graph with $\mu(G)=k$
and $\Gamma$ the $(2k+1)$-closure of $G$. By
Lemma \ref{matching-closure}, we have $\mu(\Gamma)=k$.
Let $C$ be the set of all vertices in $G'$ with
degree at least $k+1$, and let $S$ be the set of
vertices in a maximal clique that contains $C$ in $G'$.
Denote $s=|S|$. Obviously, $s\leq 2k+1$; otherwise
$\mu(\Gamma)\geq k+1$, a contradiction.
Set $f(s)=\binom{s}{2} +(2k-s+1)(n-s)$.

We claim $s=2k+1$. Suppose to the contrary. Then
$k+1\leq s\leq 2k$.
By Lemma \ref{Lem:NW} and setting $t=2k$, we have
$$e(\Gamma)\leq \max\left\{\binom{k}{2}+k(n-k),\binom{2k}{2}+n-2k\right\}.$$
As $n\leq \frac{5k+1}{2}$, we have
$$\binom{k}{2}+k(n-k)\leq 2k^2<e(G),$$
and
$$\binom{2k}{2}+n-2k\leq 2k^2-\frac{3k-1}{2}<e(G).$$
But $e(\Gamma)\geq e(G)$. This contradiction shows that $s=2k+1$.
Recall that $\mu(\Gamma)=k$. Notice that $\Gamma$ has $K_{2k+1}$ as a subgraph.
Let $S$ be such a maximum clique of $\Gamma$. If there is an edge in $E(\Gamma)\backslash E(\Gamma[S])$,
then $\mu(\Gamma)\geq k+1$, a contradiction. Hence $\Gamma$ consists of a $K_{2k+1}$ (possibly together with
$(n-2k-1)$ isolated vertices).
Obviously, $G\subseteq \Gamma$. The proof is complete.
\end{proof}

\noindent
{\bf Proof of Theorem \ref{Th:Local-Matching}.}
If $\mu=1$, then $G$ is a star or a triangle (together with some isolated vertices). One can easily check that the assertion is true. If $\mu=2$, then $G$ is a subgraph of one graph in $\{K_5\cup\overline{K_{n-5}}, 2K_3\cup\overline{K_{n-6}},K_1\vee(K_3\cup\overline{K_{n-4}}), K_2\vee\overline{K_{n-2}}\}$ (see Theorem \ref{ThMatchingFactorCritical}). One can check that the assertion is true. So we assume $\mu\geq 3$. For every edge $uv\in E(G)$, clearly $G-\{u,v\}$ has matching number $\mu-1$ or $\mu-2$. This implies that every edge $e$ has either $\mu(e)=\mu$ or $\mu(e)=\mu-1$. Set $F=\{e\in E(G): \mu(e)=\mu-1\}$. We have
\begin{align}
\sum_{e\in E(G)}\frac{1}{\mu(e)}=\frac{1}{\mu-1}|F|+\frac{1}{\mu}(e(G)-|F|)=\frac{1}{\mu}e(G)+\frac{1}{\mu(\mu-1)}|F|.
\end{align}

Let $M$ be a maximum matching of $G$. For every vertex $x$ saturated by $M$, we will denote by $\pi(x)$ the vertex matched to $x$ in $M$. 

\setcounter{claim}{0}
\begin{claim}\label{Clmuxy}
If $xy\in F$, then both $x,y$ are saturated by $M$ and $\pi(x)\pi(y)\notin E(G)$.
\end{claim}

\begin{proof}
If $x$ is not saturated by $M$, then $y$ must be saturated by $M$ since $M$ is a maximum matching. In this case, $M\backslash\{y\pi(y)\}\cup\{xy\}$ is a maximum matching of $G$, which implies that $\mu(xy)=\mu$. So we conclude that $x$, and similarly $y$, is saturated by $M$. Since $\mu(xy)=\mu-1$, we see that $xy\notin M$. If $\pi(x)\pi(y)\in E(G)$, then $M\backslash\{x\pi(x),y\pi(y)\}\cup\{xy,\pi(x)\pi(y)\}$ is a maximum matching of $G$, implying that $\mu(xy)=\mu$. So we conclude that $\pi(x)\pi(y)\notin E(G)$.
\end{proof}

We distinguish the following two cases.

\medskip\noindent\textbf{Case A.} $2\mu\leq n\leq(5\mu+1)/2$.

By Claim \ref{Clmuxy}, every edge $xy\in F$ has its two vertices saturated by $M$ and $\pi(x)\pi(y)\notin E(G)$. If $n=2\mu$, then $e(G)+|F|\leq{n\choose 2}$. It follows that
$$\sum_{e\in E(G)}\frac{1}{\mu(e)}=\frac{1}{\mu}e(G)+\frac{1}{\mu(\mu-1)}|F|
\leq\frac{1}{\mu}(e(G)+|F|)\leq\frac{1}{\mu}{n\choose 2}=n-1,$$
and the equality holds if and only if $F=\emptyset$ and $e(G)={n\choose 2}$. That is, $G\cong K_n$. Now we assume that $2\mu+1\leq n\leq(5\mu+1)/2$.

By Claim \ref{Clmuxy}, for each two edges $x\pi(x),y\pi(y)\in M$, at most one edge of $xy,\pi(x)\pi(y)$ is in $F$ and at most one edge of $x\pi(y),\pi(x)y$ is in $F$. It follows that $|F|\leq 2{\mu\choose 2}=\mu(\mu-1)$. If $e(G)\leq 2\mu^2$, then
$$\sum_{e\in E(G)}\frac{1}{\mu(e)}=\frac{1}{\mu}e(G)+\frac{1}{\mu(\mu-1)}|F|
\leq\frac{2\mu^2}{\mu}+\frac{\mu(\mu-1)}{\mu(\mu-1)}=2\mu+1,$$
as desired. Moreover, the equality holds only if $|F|=\mu(\mu-1)$ and $e(G)=\mu^2$. This implies that all edges $G[V(M)]$ is either in $F$ or in $M$. Specially $e_G(V(M))=|M|+|F|=\mu^2$. Let $U=V(G)\backslash V(M)$. So $|U|=n-2\mu\leq(\mu+1)/2$. Note that $E_G(U)=\emptyset$, implying that $e_G(V(M),U)=\mu^2$. For any edge $x\pi(x)\in M$, $x$ and $\pi(x)$ cannot neighbor two distinct vertices in $U$, which implies that $x\pi(x)$ is adjacent to at most $|U|\leq(\mu+1)/2$ edges in $E_G(V(M),U)$. Thus $e_G(V(M),U)\leq\mu(\mu+1)/2<\mu^2$, a contradiction.

Now we assume that $e(G)>2\mu^2$. By Lemma \ref{Lem:Matching-Stab}, $G$ is a subgraph of $K_{2\mu+1}\cup\overline{K_{n-2\mu-1}}$. Notice that for every $xy\in F$, $\pi(x)\pi(y)\notin E(G)$. It follows that $e(G)+|F|\leq{2\mu+1\choose 2}$. Thus
$$\sum_{e\in E(G)}\frac{1}{\mu(e)}=\frac{1}{\mu}e(G)+\frac{1}{\mu(\mu-1)}|F|\leq\frac{1}{\mu}(e(G)+|F|)\leq\frac{1}{\mu}{2\mu+1\choose 2}=2\mu+1,$$ as desired. Moreover, the equality holds if and only if $F=\emptyset$ and $e(G)={2\mu+1\choose 2}$. That is, $G\cong K_{2\mu+1}\cup\overline{K_{n-2\mu-1}}$.

\medskip\noindent\textbf{Case B.} $n\geq 5\mu/2+1$.

By Theorem \ref{ThMatchingFactorCritical}, there is a set $S\subseteq V(G)$ such that every component of $G-S$ is factor-critical and $S$ is matched to $\mathcal{C}_{G-S}$. Let $S=\{x_1,\ldots,x_s\}$, and let $H_1,\ldots,H_t$ be the components of $G-S$. Set $h_i=|V(H_i)|$, where $i=1,\ldots,t$.

Recall that $M$ is a maximum matching. We have that $M$ consists of $s$ edges from $S$ to different components of $G-S$, and nearly perfect matchings for all $H_i$, $i\in[1,t]$. We have that $n-2\mu=t-s$. Without loss of generality, we can assume that $\pi(x_i)\in V(H_i)$, $i\in[1,s]$. 

Let $G'$ be the graph obtained from $G$ by adding all missing edges in
$$\{xx': x,x'\in S\}\cup\{xy: x\in S, y\in V(G)\backslash S\}\cup\bigcup_{i=1}^t\{yy': y,y'\in V(H_i)\}.$$
 We remark that each $S\cup V(H_i)$ is a clique of $G'$. We set $A=E(G')\backslash E(G)$ the set of edges added in $G'$.
 Notice that $H_i$ is factor-critical, every edge $x_iy_i$ with $y_i\in V(H_i)$ is contained in a maximum matching of $G$, i.e., $x_iy_i\notin F$.

 \begin{claim}\label{ClFiAi}
 (1) $F\cap E_G(S,H_i)=\emptyset$ for all $i\in[s+1,t]$.\\
 (2) Let $F'_i=F\cap E_G(S,V(H_i))$, $i\in[1,s]$, and $A'_i=A\cap E_{G'}(x_i, V(G)\backslash(S\cup V(H_i)))$. We have $|F'_i|\leq(\mu-1)|A'_i|$, and the equality holds if and only if $F'_i=A'_i=\emptyset$.\\
 (3) Let $F_i=F\cap E(H_i)$, and $A_i=A\cap E_{G'}(\{x_i\}\cup V(H_i))$ if $i\in[1,s]$; $A_i=A\cap E_{G'}(V(H_i))$ if $i\in[s+1,t]$. We have $|F_i|\leq|A_i|$.
 \end{claim}

 \begin{proof}
 (1) Let $x_iy_j\in E(G)$ with $i\in[1,s]$, $j\in[s+1,t]$. We will show that $x_iy_j\notin F$. Let $M_j=M\cap E_G(H_j)$. Since $H_j$ is factor-critical, $H_j-y_j$ has a perfect matching $M'_j$. Now $M'=(M\backslash(M_j\cup\{x_i\pi(x_i)\})\cup(M'_j\cup\{x_iy_j\})$ is also a maximum matching of $G$. This implies $x_iy_j\notin F$.

 (2) If $F'_i=\emptyset$, then the assertion is trivial. Now suppose $F'_i\neq\emptyset$. If $x_i$ has a neighbor $y_j$ in $H_j$ with $j\in[s+1,t]$, then $G$ has a maximum matching $M'$ with $x_iy_j\in M'$ and $M'\cap E_G(S,V(H_i))=\emptyset$. From (1) we see that $F'_i=\emptyset$, a contradiction. So we conclude that $x_i$ is not adjacent to all vertices in $H_j$ with $j\in[s+1,t]$. Thus, $|A'_i|\geq t-s=n-2\mu\geq\mu/2+1$.


 Let $S_0=\{x_j\in S: x_j\mbox{ is incident to some edge in }F'_i\}$. Then $x_i\notin S_0$ and $F'_i\subseteq E_G(S_0,V(H_i))$. It follows that $|F'_i|\leq s_0h_i$, where $s_0=|S_0|$. Notice that $\mu=|M|\geq(h_i+1)/2+s_0>h_i/2+s_0$. It follows that $|F'_i|\leq s_0h_i<\mu^2/2$ (using mean value inequality).

 Now, it suffices to show that $\mu^2/2<(\mu-1)(\mu/2+1)$, which is true when $\mu\geq 3$.
 
 (3) Notice that for every edge $yy'\in F_i$, $\pi(y),\pi(y')\in\{x_i\}\cup V(H_i)$ if $i\in[1,s]$, and $\pi(y),\pi(y')\in V(H_i)$ if $i\in[s+1,t]$. By Claim \ref{Clmuxy}, we see that $\pi(y)\pi(y')\notin E(G)$, and thus the assertion $|F_i|\leq|A_i|$ follows. 
 \end{proof}

 We define $F_i,A_i,F'_i,A'_i$ as in Claim \ref{ClFiAi}. Recall that all edges from $x_i$ to $H_i$ or $H_j$ with $j\in[s+1,t]$, are not in $F$. It follows that $F=E(G[S])\cup\bigcup_{i=1}^tF_i\cup\bigcup_{i=1}^sF'_i$. Now we have
 \begin{align*}
 \sum_{e\in E(G)}\frac{1}{\mu(e)} & =\frac{1}{\mu}e(G)+\frac{1}{\mu(\mu-1)}|F|\\
    & =\frac{1}{\mu}e(G)+\frac{1}{\mu(\mu-1)}\left(e(G[S])+\sum_{i=1}^t|F_i|+\sum_{i=1}^s|F'_i|\right)\\
    & \leq\frac{1}{\mu}e(G)+\frac{1}{\mu(\mu-1)}\left(e(G[S])+\sum_{i=1}^t|A_i|+(\mu-1)\sum_{i=1}^s|A'_i|\right)\\
    & \leq\frac{1}{\mu}\left(e(G)+\sum_{i=1}^t|A_i|+\sum_{i=1}^s|A'_i|\right)+\frac{1}{\mu(\mu-1)}e(G[S])\\
    & \leq\frac{1}{\mu}e(G')+\frac{1}{\mu(\mu-1)}e(G'[S]).
 \end{align*}
Moreover, the equality holds if and only if $F_i=A_i=\emptyset$ for all $i\in[1,t]$, $F'_i=A'_i=\emptyset$ for all $i\in[1,s]$, and $e(G)=e(G')$, implying that $G=G'$.

Let $G''$ be the graph obtained from $G'$ by removing $h_i-1$ vertices from $H_i$ to $H_1$ for all $i\in[2,t]$, i.e., $G''\cong K_s\vee(K_{n-s-t+1}\cup\overline{K_{t-1}})$. Then $e(G')\leq e(G'')$, and
 \begin{align*}
 \frac{1}{\mu}e(G')+\frac{1}{\mu(\mu-1)}e_{G'}(S) 
    & \leq\frac{1}{\mu}e(G'')+\frac{1}{\mu(\mu-1)}e_{G''}(S)\\
    & =\frac{1}{\mu}\left(s(n-s)+{2\mu-2s+1\choose 2}\right)+\frac{1}{\mu-1}{s\choose 2}
 \end{align*}
(recall that $n-2\mu=t-s$). Note that $0\leq s\leq\mu$. The above formula takes the maximum value when $s=0$ or $s=\mu$. It follows that
 \begin{align*}
 \frac{1}{\mu}e(G'')+\frac{1}{\mu(\mu-1)}e_{G''}(S) 
    & \leq\max\left\{\frac{1}{\mu}{2\mu+1\choose 2},\frac{1}{\mu}\mu(n-\mu)+\frac{1}{\mu-1}{\mu\choose 2}\right\}\\
    & =\max\left\{2\mu+1,n-\frac{\mu}{2}\right\}.
 \end{align*}
with equality if and only if $n=5\mu/2+1$ and $G''\cong K_{2\mu+1}\cup\overline{K_{n-2\mu-1}}$, or $n\geq 5\mu/2+1$ and $G''\cong K_{\mu}\vee\overline{K_{n-\mu}}$. This proves the theorem.
\qed

\vspace{2mm}
\noindent
{\bf Proof of Theorem \ref{Thm:Local-generalized-Turan}.}
(i) We prove the theorem by induction on $s$. The case $s=2$ is
Theorem \ref{Theorem:MT}. 
Suppose the proposition is true for $s$. Now we prove the case of $s+1$. For any vertex $x\in V(G)$, we denote
$G_x$ by the subgraph of $G$ induced by $N_G(x)$, and by $p(S,G_x)$
the maximum length of a path $P$ in $G_x$ such that $S\subset V(P)$
and vertices of $S$ appear in $P$ in the consecutive order. Then, we
have
$$\sum_{x\in V(G)}\sum_{S\in N_s(G_x)}\frac{1}{p(S,G_x)-s+2}\leq \sum_{x\in V(G)}\frac{n_{s-1}(G_x)}{s}.$$
Observe that $\sum_{x\in V(G)}n_{s-1}(G_x)=sn_s(G)$ and $\sum_{x\in V(G)}n_{s}(G_x)=(s+1)n_{s+1}(G)$.

We claim $p(S,G)\geq p(S,G_x)+1$. Set $t=p(S,G_x)$. In fact, assume that $P=v_1v_2\ldots v_t$ is the required
path of $G_x$ such that $T=\{v_i,v_{i+1},\ldots,v_{i+t}\}$ appears on $P$ in the consecutive order and
$T$ induces a clique of size $s$. Then $P'=xv_1\cup P$ is a path of length $|P|+1$ such that $T$ appears
on $P'$ in the consecutive order and induces a $K_s$. It follows $p(S,G)\geq p(S,G_x)+1$.

Thus, we have
\begin{align*}
&\sum_{x\in V(G)}\sum_{S\in N_s(G_x)}\frac{1}{p(S,G_x)-s+2}\\
&\geq \sum_{x\in V(G)}\sum_{S\in N_s(G_x)}\frac{1}{p(S,G)-s+1}\\
&=\sum_{S\in N_{s+1}(G)}\frac{s+1}{p(S,G)-s+1}.
\end{align*}
Hence, together with the above two inequalities, we have
$$
\sum_{S\in N_{s+1}(G)}\frac{s+1}{p(S,G)-s+1}\leq n_s(G),
$$
which implies that
$$
\sum_{S\in N_{s+1}(G)}\frac{1}{p(S,G)-s+1}\leq \frac{n_s(G)}{s+1}.
$$
The proof is complete.

(ii) Before the proof, we first prove a lemma.
\begin{lem}
Let $G$ be a graph with $\omega=\omega(G)$. Then for any $1\leq s\leq \omega(G)$,
we have
$$\triangle(G)\geq \frac{(s+1)n_{s+1}(G)}{n_s(G)}+s-1.$$
\end{lem}
\begin{proof}
For any $s$-clique $S$, the number of $(s+1)$-cliques through $S$ equals to $\cap_{x\in S}N_G(x)$.
Note that $\cap_{x\in S}N_G(x)\leq |N_G(x)\backslash (S-\{x\})|\leq \triangle(G)-s+1$.
Thus, $n_{s+1}(G,S)\leq \triangle(G)-s+1$. It follows that $\sum_{S\in N_s(G)}n_{s+1}(G,S)=(s+1)n_{s+1}(G)\leq n_s(G)(\triangle(G)-s+1)$.
Hence, $\triangle(G)\geq \frac{(s+1)n_{s+1}(G)}{n_s(G)}+s-1$. The proof is complete.
\end{proof}

The proof of Theorem \ref{Thm:Local-generalized-Turan}(B) is similar to that
of Theorem \ref{Thm:Local-generalized-Turan}(A). The only new part is the following inequality. We omit the details.
\begin{lem}
Let $G$ be a graph and $x\in V(G)$. Then we have
$$
s(K,G_x)+1\leq s(K,G).
$$
\end{lem}

The case $s=2$ is the following proposition (see Proposition 1 in \cite{MT2023}).
\begin{proposition}[Malec-Tompkins, 2023 \cite{MT2023}]
Let $G$ be any $n$-vertex graph, then
$$
\sum_{e\in E(G)}\frac{1}{s(e)}\leq \frac{n}{2},
$$
where $s(e)$ denotes the size of a maximum star containing the edge $e$.
\end{proposition}
\qed
\begin{remark}
The Malec-Tompkins Theorem is equivalent to the following inequality
$$
\sum_{uv\in E(G)}\frac{1}{\max\{d(u),d(v)\}}\leq \frac{n}{2}.
$$

\end{remark}

\section*{Acknowledgments}
Part of this work was completed when the authors  were visiting Alfr\'{e}d R\'{e}nyi Institute of Mathematics during
July 2024. The authors thank Professor Ervin Gy\H{o}ri for his warm hospitality.
The authors are grateful to Ervin Gy\H{o}ri, Yandong Bai, and Chuanqi Xiao for helpful discussions
on the proofs of Theorems \ref{Th:Local-BBRS} and \ref{Th:Local-Matching}. The second author also thanks Mingzhu Chen, Zhidan Luo, and Zhenyu Ni for helpful discussions on the short proof
of Theorems \ref{Theorem:MT} when he was visiting Hainan University
in April 2024. Finally, the second author is grateful to Tompkins for drawing his attention to the paper \cite{MT2023}.

\end{document}